\documentclass{amsart}
\newcommand{\R}{\mathbb{R}}
\newcommand{\C}{\mathbb{C}}
\newcommand{\Z}{\mathbb{Z}}
\newcommand{\N}{\mathbb{N}}
\newcommand{\Q}{\mathbb{Q}}
\newtheorem{Theo}{Theorem}

\newtheorem{coro}{Corollary}
\newtheorem{Lem}{Lemma}
\newtheorem{Con}{Conjecture}
\theoremstyle{definition}
\newtheorem{ex}{Example}

\title[Essential singularities]{Essential singularities of
  Euler products}
\author{Gautami Bhowmik}
\author{Jan-Christoph Schlage-Puchta}
\begin{document}
\renewcommand{\subjclass}{\textup{2010} Mathematics Subject Classification}
\renewcommand{\thefootnote}{}
\footnote{\subjclass{ 30B50, 11M41, 30B40, 20F69, 11G50 }} 
\begin{abstract} We classify singularities of Dirichlet series having 
Euler products which are rational functions of $p$ and $p^{-s}$ for $p$ a prime
number and give examples of natural boundaries from zeta functions of
groups and height zeta functions.
\end{abstract}
\maketitle

\section{Introduction and results}

Many Dirichlet-series occurring in practice satisfy an Euler-product,
and if they do so, the Euler-product is often the easiest way to
access the series. Therefore, it is important to deduce information on
the series from the Euler-product representation. One of the most
important applications of Dirichlet-series, going back to Riemann, is
the asymptotic estimation of the sum of its coefficients via Perron's
formula, that is, the use of the equation
\[
\sum_{n\leq x} a_n = \frac{1}{2\pi i}\int\limits_{c-i\infty}^{c+i\infty}
\Big(\sum_{n\geq 1}\frac{a_n}{n^s}\Big)\frac{x^s}{s}\;ds.
\]
To use this
relation, one usually shifts the path of integration to the left,
thereby reducing the contribution of the term $x^s$. This becomes 
possible only if the function $D(s)=\sum\frac{a_n}{n^s}$ is holomorphic on the
new path and therefore the question of continuation of Dirichlet-series
beyond their domain of absolute convergence is a central issue in this
theory. In fact, the importance of the Riemann hypothesis stems from
the fact that it would allow us to move the path of integration for
$D(s)=\frac{\zeta'}{\zeta}(s)$ to the line $1/2+\epsilon$ without
meeting any singularity besides the obvious pole at 1. 

Estermann\cite{Est} appears to be the first to address this problem. He showed
that for an integer valued polynomial $W(x)$
with $W(0)=1$ the
Dirichlet-series $D(s)=\prod_p W(p^{-s})$ can either be written as a
finite product of the form $\prod_{\nu\leq N}\zeta(\nu s)^{c_\nu}$ for
certain integers $c_\nu$, and is therefore meromorphically continuable
to the whole complex plane, or is continuable to the half-plane
$\Re\,s>0$. In the latter case the line $\Re\,s=0$ is the natural
boundary of 
the Dirichlet-series.  
The strategy of his proof was to show that every point on the line
$\Re\,s=0$ is an accumulation point of poles or zeros of $D$. Note
that $\zeta$, the Riemann-zeta function itself, does not fall among the cases under
consideration, since  $W(X)=(1-X)^{-1}$ is a rational
function. Dahlquist\cite{Dahl} generalized Estermann's work allowing
$W$ to be a function holomorphic in the unit circle with the exception
of isolated singularities and in particular covering the case that $W$
be rational. This
method of proof was extended to much greater generality, interest
being sparked by $\zeta$-functions of nilpotent groups
introduced by Grunewald, Segal and Smith\cite{GSS} as well as height zeta
functions\cite{2nob}. Functions arising
in these contexts are often of the form $D(s)=\prod W(p, p^{-s})$ for an
integral polynomial $W$. Du Sautoy and Grunewald\cite{ghost} gave a
criterion for such a function to have a natural boundary which, in a
probabilistic sense, applies to almost all polynomials. 
Again, it is shown that every point on
the presumed boundary is an accumulation point of zeros or
poles.  The following conjecture, see for example \cite[\bf 1.11]{book}\cite[\bf 1.4]{ghost}, 
is believed to be true.
\begin{Con}
\label{Con:Main}
Let $W(x, y)=\sum_{n,m} a_{n,m} x^n y^m$ be an integral polynomial
with $W(x, 0) = 1$. Then $D(s)=\prod_p W(p, p^{-s})$ is
meromorphically continuable to the whole complex plane if and if only
if it is a finite product of Riemann $\zeta$-functions. Moreover, in
the latter case if $\beta=\max\{\frac{n}{m}: m\ge 1, a_{n,m}\neq 0\}$, then
$\Re\,s=\beta$ is the natural boundary of $D$.
\end{Con}
In this paper we show that any refinement of Estermann's method is bound to
fail to prove this conjecture. 

If $W(X, Y)$ is a rational function, we expand $W$ into a power
series $W(X, Y)=\sum_{n, m\geq 0} a_{n, m} X^nY^m$, and define
$\alpha=\sup\{\frac{n+1}{m}: m\ge 1,a_{n,m}\neq 0\}$,
$\beta=\sup\{\frac{n}{m}: m\ge 1,a_{n,m}\neq 0\}$. It is easy to see that 
the supremum is actually attained, and that the function
$\tilde{W}=1+\sum_{\frac{n}{m}=\beta} a_{n, m} X^nY^m$ is again a
rational function. We call
$\tilde{W}$ the main part of $W$, since only 
$\tilde{W}$ is responsible for the convergence of the product
$D(s)$. For $W$ a polynomial $\tilde{W}$ was called the ghost of $W$
in \cite{ghost}. A rational function $W$ is called cyclotomic if it
can be written as the product of cyclotomic polynomials and their
inverses. 

We define an obstructing point $z$ to be a complex number with
$\Re\,z=\beta$, such that there exists a sequence of complex numbers
$z_i$, $\Re\,z_i>\beta$, $z_i\rightarrow z$, such that $D$ has a pole
or a zero in $z_i$ for all $i$. Obviously, each obstructing point is
an essential singularity for $D$, the converse not being true in
general. 

Our main result is the following.

\begin{Theo}
\label{thm:class}
Let $W(X, Y)$ be a rational function, which can be written as
$\frac{P(X, Y)}{Q(X, Y)}$, where $P, Q\in\Z[X, Y]$ satisfy $P(X,
0)=Q(X, 0)=1$. Define $a_{n, m}, \beta, \tilde{W}$ and $D$ as
above. Then the product representation of $D$ converges in the
half-plane $\Re\;s>\alpha$, $D$ 
can meromorphically continued into the half-plane $\Re\;s>\beta$, and
precisely one of the following holds true.
\begin{enumerate}
\item $W$ is cyclotomic and once its unitary factors are removed,
$W=\tilde{W}$ ; in this case $D$ is a
  finite product of Riemann $\zeta$-functions;
\item $\tilde{W}$ is not cyclotomic; in this case every point of the
  line $\Re\,s =\beta$ is an obstructing point;
\item\label{type2} 
$W\neq\tilde{W}$, $\tilde{W}$ is cyclotomic and there are
  infinitely many pairs $n,m$ with $a_{n,m}\neq 0$ and
  $\frac{n}{m}<\beta<\frac{n+1}{m}$; in this case $\beta$ is an
  obstructing point;
\item $W\neq\tilde{W}$, $\tilde{W}$ is cyclotomic, there are only
  finitely many pairs $n,m$ with $a_{n,m}\neq 0$ and
  $\frac{n}{m}<\beta<\frac{n+1}{m}$, but there are infinitely many
  primes $p$ such that the equation $W(p, p^{-s})=0$ has a solution
  $s_0$ with $\Re\,s_0>\beta$; in this case every point of the line
  $\Re\,s=\beta$ is an obstructing point;
\item None of the above; in this case no point on the line
  $\Re\,s=\beta$ is an obstructing point.
\end{enumerate}
\end{Theo}
We remark that each of these cases actually occurs, that is, there are
Euler-products for which Estermann's approach cannot work.

Notice that while in the third case we need information on the zeros of the
Riemann-zeta function to know about the meromorphic continuation, in
the last case we can say nothing about their continuation.

While the above classification looks pretty technical, these cases
actually behave quite differently. To illustrate this point we consider
 a domain $\Omega\subseteq\C$ with a function
$f:\Omega\rightarrow\C$, let $N_\pm(\Omega)$  the number of zeros
and poles of $f$ in $\Omega$ counted with positive multiplicity, that
is, an $n$-fold zero or a pole of order $n$ is counted $n$
times. Then we have the following.
\begin{coro}
Let $W$ be a rational function, and define $\beta$ as above. Then one of the
following two statements holds true:
\begin{enumerate}
\item For every $\epsilon>0$ we have $N_\pm(\{|z-\beta|<\epsilon, \Re
  z>0\}) = \infty$;
\item We have $N_\pm(\{\Re z>\beta, |\Im z|<T\} =
  \mathcal{O}(T\log T)$. 
\end{enumerate}
If  $\ W$ is a polynomial and  we assume the Riemann
  hypothesis as well as the $\Q$-linear independence of the imaginary parts
  of the non-trivial zeros of $\zeta$, then there exist constants $c_1,
  c_2$, such that $N_\pm(\{\Re z>\beta, |\Im z|<T\}) =
  c_1T\log T + c_2 T + \mathcal{O}(\log T)$.

\end{coro}

Finally we remark that for $\zeta$-functions of nilpotent groups the
generalization to rational functions is irrelevant, since a result of du
Sautoy\cite{Denominator} implies that if $\zeta_G(s)=\prod_p W(p,
p^{-s})$ for a rational function $W(X, Y)=\frac{P(X, Y)}{Q(X, Y)}$,
then $Q$ is a cyclotomic polynomial, that is, $\zeta_G$ can be written
as the product of finitely many Riemann $\zeta$-functions and a
Dirichlet-series of the form $\prod_p W(p, p^{-s})$ with $W$ a
polynomial. However, for other applications it is indeed important to
study rational functions, one such example occurs in the recent work of de
la Bret\`eche and Swynnerton-Dyer\cite{2nob}.

\section{Proof of case 2}

In this section we show that if $\tilde{W}$ is not cyclotomic, then
$\Re\;s=\beta$ is the natural boundary of the meromorphic
continuation of $D$. For $W$ a polynomial this was shown by du Sautoy
and Grunewald\cite{ghost}, our proof closely follows their lines of
reasoning. 

The main difference between the case of a polynomial and a rational function
is that for polynomials the local zeros created by different primes can never
cancel, whereas for a rational function the zeros of the numerator belonging
to some prime number $p$ might coincide with zeros of the denominator
belonging to some other prime $q$, and may therefore not contribute to the
zeros or poles needed to prove that some point on the presumed boundary is a
cluster point. We could exclude the possibility of cancellations by assuming
some unproven hypotheses from transcendence theory, however, here we show that
we can deal with this case unconditionally by proving that the amount of
cancellation remains limited. We first consider the case of cancellations
between the numerator and denominator coming from the same prime number.

\begin{Lem}
\label{Lem:alggeo}
Let $P, Q\in\Z[X, Y]$ be co-prime non-constant polynomials. Then there are only
finitely many primes $p$, such that for some complex number $s$ we have $P(p,
p^{-s})=Q(p, p^{-s})=0$.
\end{Lem}
\begin{proof}
Let $V$ be the variety of $\langle P, Q\rangle$ over $\C$. Assume there are
infinitely many pairs $(p, s)$, for which the equation $P(p,
p^{-s})=Q(p, p^{-s})=0$ holds true. Then $V$ is infinite, hence, at least
one-dimensional. Since $P$ and $Q$ are non-constant, we have $V\neq \C^2$,
hence, $V$ is one-dimensional. Let $V'$ be a one-dimensional irreducible
component, and let $R$ be a generator of the ideal corresponding to $V'$. Then
$\langle P, Q\rangle\subseteq\langle R\rangle$, that is, $R$ divides $P$ and
$Q$, which implies that $R$ is constant. But a constant polynomial cannot
define a one-dimensional variety and this contradiction proves our claim.
\end{proof}

Next we use the following graph-theoretic result, describing graphs which are
rather close to trees. We call a cycle in a graph {\it minimal}\/ if it is of length
$\geq 3$, and not the union of two cycles of smaller length.
\begin{Lem}
\label{Lem:Graph}
Let $\mathcal{G}$ be a graph, $k\geq 2$ an integer, such that every vertex has
degree $\geq 3k$, and that there exists a symmetric relation $\sim$ on the
vertices, such that every vertex $v$ is in relation to at most $k$ other
vertices, and every minimal cycle passing through $v$ also passes
through one of the vertices in relation to $v$. Then $\mathcal{G}$ is
infinite.
\end{Lem}
\begin{proof}
Suppose that $\mathcal{G}$ were finite, and fix some vertex $v_0$. We call a
geodesic path {\it good}\/ if no two vertices of the path 
stand in relation to each other. We want to construct an infinite good
path. Note that 
$p_1$ and $p_2$ are good paths of finite length, they cannot intersect in but
one point, for otherwise their union would contain a cycle, and choosing one
of the intersection points we would obtain a contradiction with the
definition of a
good path. Hence, the union of the good paths starting in $v_0$ forms a tree. There
are $\geq 3k$ vertices connected to $v_0$, at most $k$ of which are forbidden.
Hence, the first layer of the tree contains at least $2k$ points. Each of
these points is connected to at least $3k$ other points. It stands in relation
to at most $k$ of them and hence we can extend every path in at least $2k$ ways,
and of all these paths at most $k$ stand in relation with $v_0$. Hence, the
second layer contains at least $4k²-k$ points. Denote by $n_i$ the number of
points in the $i$-th layer of the tree. Then, continuing in this way, we
obtain
\[
n_{i+1}\geq 2kn_i - k(n_{i-1} + \dots + n_1 + 1).
\]
From this and the assumption that $k\geq 2$ it follows by induction that
$n_{i+1}\geq kn_i$, hence, the tree and therefore the graph $\mathcal{G}$,
which contains the tree, is infinite.
\end{proof}

Note the importance of symmetry: if the relation is allowed to be
non-symmetric, we can get two regular trees, and identify their leaves. Then every
minimal cycle passing through one point either passes through its parent node
or the mirror image of the point. Thus in the absence of symmetry the result becomes
wrong for arbitrarily large valency even for $k=2$.

We can now prove our result on non-cancellation.

\begin{Lem}
\label{Lem:nocancel}
Let $P, Q\in\Z[X, Y]$ be co-prime polynomials with $\beta$ defined as in the
introduction. Let $\epsilon>0$ be given, and suppose that for a prime $p_0$
sufficiently large $P(p_0, p_0^{-s})$ has a zero on the segment $[\sigma+it,
\sigma+it+\epsilon]$, where $\sigma>\beta$. Then $\prod_p\frac{P(p,
  p^{-s})}{Q(p, p^{-s})}$ has a zero or a pole on this segment.
\end{Lem}
\begin{proof}
Since the local zeros converge to the line $\Re\;s=\beta$, there are only
finitely many primes $p$ for which the numerator or denominator has a zero,
hence, we may assume that $P(p, p^{-s}), Q(p, p^{-s})\neq 0$ for $p>p_0$. For
each prime $p$ let $z_1^p, \ldots, z_k^p$ be the roots of the equation
$P(p,p^{-s})=0$ in the segment $\Re\;s=\beta$, $0\leq\Im\;s\leq
\frac{2\pi}{\log p}$, and let $w_1^p, \ldots, w_\ell^p$ be the roots of the
equation $Q(p,p^{-s})=0$ on this segment. Such roots need not exist  but
if they do then their number is bounded independently of
$p$. The roots of the equations $P(p, p^{-s})=0$ and $Q(p, p^{-s})=0$ form a
pattern with period $\frac{2\pi i}{\log p}$. If $p_0$ is sufficiently large,
then $\delta$ becomes arbitrary small, hence, if $p$ is not large then
the equations $P(p,p^{-s})=0$ and $Q(p, p^{-s})=0$ do not have
solutions on the 
line $\Re\;s=\beta+\delta$. Let $p_1$ be the least prime for which such solutions
exist. For $p_1$ sufficiently large and $p>p_1$, either $P(p, p^{-s})=0$  has
no solution on the segment under consideration or it has at least
$\big[\frac{\epsilon\log p}{2\pi}\big]$ such solutions. Note that by fixing
$\epsilon$ and choosing $p_0$ sufficiently large we can make this expression
as large as we need. Further note that by choosing $p_0$ large we can ensure,
in view of Lemma~\ref{Lem:alggeo}, that $P(p, p^{-s})=Q(p, p^{-s})=0$ has no
solution on the line $\Re\;s=\beta+\delta$.

We now define a bipartite graph $\mathcal{G}$ as follows: The vertices of the
graph are all complex numbers $z_i^p$ in one set and all complex numbers
$w_i^p$ in the other set, where $p\leq p_0$. Two vertices $z_i^p$, $w_j^q$
are joined by an edge if there exists a complex number $s$ with
$\Re\;s=\beta+\delta$, $t\leq\Im\; s\leq t+\epsilon$, such that $s$ is
congruent to $z_i^p$ modulo $\frac{2\pi i}{\log p}$, and congruent to $w_j^q$
modulo $\frac{2\pi i}{\log q}$. In other words, the existence of an edge
indicates that one of the zeros of $P(p, p^{-s})$ obtained from $z_i^p$ by
periodicity cancels with one zero of $Q(q, q^{-s})$ obtained from $w_j^q$.
If $\prod_p\frac{P(p,p^{-s})}{Q(p, p^{-s})}$ has neither a zero nor a pole on
the segment, then every zero of one of the polynomials cancels with a zero of
the other polynomial, that is, every vertex has valency at least
$\big[\frac{\epsilon\log p_1}{2\pi}\big]$. 

We next bound the number of cycles. Suppose that $z_{i_1}^{p_1}\sim
w_{i_2}^{p_2}\sim\dots\sim w_{i_\ell}^{p_\ell}\sim z_{i_1}^{p_1}$. Then there
is a complex number $s$ in the segment which is congruent to
$z_{i_1}^{p_1}$ modulo $\frac{2\pi i}{\log p_1}$ and congruent to
$w_{i_2}^{p_2}$ modulo $\frac{2\pi i}{\log p_2}$. Going around  the
cycle and collecting the differences we obtain an equation of the form $2\pi
i\sum_i \frac{\lambda_i}{\log p_i} = 0$, $\lambda_i\in\Z$, which can only hold
if the combined coefficients vanish for each occurring prime. However the
coefficients cannot vanish if some prime occurs only once. If the cycle is
minimal the same vertex cannot occur twice, hence, there is some $j$ such
that $p_1=p_j$, but $i_1\neq i_j$. Hence, every minimal cycle containing
$z_{i_1}^{p_1}$ must contain $z_{i_j}^{p_j}$ or $w_{i_j}^{p_j}$ for some $i\neq
j$. The relation defined by $x_i^p\sim x_j^q\Leftrightarrow p=q$, $x\in\{z,
w\}$ is an equivalence relation with equivalence classes bounded by some
constant $K$. If we choose $p_1>\exp(6\pi K\epsilon^{-1})$, the assumptions of
Lemma~\ref{Lem:Graph} are satisfied, and we conclude that $\mathcal{G}$ is
finite.

But we already know that there is no $p>p_0$ for which $P(p,p^{-s}) = 0$ or
$Q(p, p^{-s}) = 0$ has a solution, that is, $\mathcal{G}$ is finite. This
contradiction completes the proof.
\end{proof}

Using Lemma~\ref{Lem:nocancel} the proof now proceeds in the same fashion as
in the polynomial case; for the details we refer the reader to the proof given
by du Sautoy and Grunewald\cite{ghost}. 
\section{Development in cyclotomic factors}
A rational function $W(X, Y)$ with $W(X, 0)=1$ can be written as an
infinite product of polynomials of the form $(1-X^aY^b)$. Here
convergence is meant with 
respect to the topology of formal power series, that is, a product
$\prod_{i=1}^\infty(1-X^{a_i}Y^{b_i})$ converges to a power series $f$
if for each $N$ there exists an $i_0$, such that for $i_1>i_0$ the
partial product $\prod_{i=1}^{i_1}(1-X^{a_i}Y^{b_i})$ coincides with
$f$ for all coefficients of monomials $X^aY^b$ with $a, b<N$. The
existence of such an extension is quite obvious, however, we need some
explicit information on the  factors that occur and we shall develop the
necessary information here.

For a set $A\subseteq\R^2$ define the convex cone $\overline{A}$
generated by $A$ to be the smallest convex subset containing $\lambda
a$ for all $a\in A$ and $\lambda>1$. A point $a$ of $A$ is extremal, if it
is contained in the boundary of $\overline{A}$ and there exists a
tangent to $A$ intersecting $A$ precisely in $a$, or set theoretically
speaking, if $\overline{A\setminus\{a\}}\neq\overline{A}$. Note that a
convex cone forms an additive semi-group as a subsemigroup of $\R^2$.

To a formal power series $W=\sum_{n, m} a_{n, m} X^n Y^m\in\Z[[X, Y]]$ we
associate the set $A_W = \{(n, m): a_{n, m}\neq 0\}$. Suppose we start
with a rational function $W\in\Z[X, Y]$, which is of the form
$\frac{P(X, Y)}{Q(X, Y)}$ with $P(X, 0)=Q(X, 0)=1$. Then
\[
\frac{1}{Q(X, Y)} = \sum_{\nu=0}^\infty \big(Q(X, Y)-1\big)^\nu =
\sum_{n, m} b_{n, m} X^n Y^m,
\]
say, where the convergence of the geometric series as a formal power series
follows from the fact that every monomial in $Q$ is divisible by
$Y$. The set $\{(n, m):b_{n, m}\neq 0\}\subseteq\R^2$ is contained in
the semigroup generated by the points corresponding to monomials in
$Q$, but may be strictly smaller, as there could be unforeseen
cancellations. Multiplying the power series by $P(X, Y)$, we obtain
that $A_W$ is contained within finitely many shifted copies of
$A_{Q^{-1}}$.

Let $(n, m)$ be an extremal point of
$A_w$. Then we have $W=(1-X^n Y^m)^{-a_{n, m}} W_1(X, Y)$, where
$W_1(X, Y) = (1-X^n Y^m)^{a_{n, m}} W(X, Y)$. Obviously, $W_1(X, Y)$
is a formal power series with integer coefficients, we claim that
$\overline{A_{W_1}}$ is a proper subset of $\overline{A_W}$. In fact,
the monomials of $W_1$ are obtained by taking the monomials of $A_W$,
multiplying them by some power of $X^nY^m$, and possibly adding up the
contribution of different monomials. Hence, $A_{W_1}$ is contained
in the semigroup generated by $A_W$. But $(n, m)$ is not in $A_{W_1}$,
and since $(n, m)$ was assumed to be extremal, we obtain
\[
A_{W_1}\subseteq\langle A_W\rangle\setminus\{(n, m)\} \subseteq
\overline{A_W}\setminus\{(n, m)\} \subset\overline{A_W}.
\]
Taking the convex cone is a hull operator, thus $\overline{A_{W_1}}$
is a proper subset of $\overline{A_{W}}$. Since we begin and end with a subset
of $\N^2$, we can repeat this procedure so that after finitely many
steps the resulting power series $W_k$ contains no non-vanishing
coefficients $a_{n, m}$ with $n<N, m<M$. This suffices to prove the
existence of a product decomposition, in fact, if one is not
interested in the occurring cyclotomic factors one could avoid power
series and stay within the realm of polynomials by setting $W_1(X, Y)
= (1+X^n Y^m)^{-a_{n, m}} W(X, Y)$ whenever $a_{n, m}$ is negative.
However, in this way we trade one operation involving power series for
infinitely many involving polynomials, which is better avoided for
actual calculations.

While we can easily determine a super-set of $A_{W_1}$, in general we
cannot prove that some coefficient of $W_1$ does not vanish, that is,
knowing only $A_W$ and not the coefficients
we cannot show that $A_{W_1}$ is as large as we suspect it to
be. However, it is easy to see that when eliminating one extremal
point all other extremal points remain untouched. In particular, if we
want to expand a polynomial $W$ into a product of cyclotomic polynomials,
at some stage we have to use every extremal point of $A_W$, and the
coefficient attached to this point has not changed before this step,
by induction it follows that the expansion as a cyclotomic product is
unique. 

We now assume that $\tilde{W}$ is cyclotomic, while $W$ is not. We
further assume that $\tilde{W}$ is a polynomial, and that the
numerator of $W$ is not divisible by a cyclotomic polynomial. We can
always satisfy these assumptions by multiplying or dividing $W$ with
cyclotomic polynomials, which corresponds to the multiplying or
dividing $D$ with certain shifted $\zeta$-functions, and does not
change our problem. Our aim is to find some 
information on the set $\{(n,m):c_{n, m}\neq 0\}$, where the
coefficients $c_{n, m}$ are defined via the expansion $W(X, Y)=\prod
(1-X^n Y^m)^{c_{n, m}}$. 

In the first step we remove all points on the line
$\frac{n}{m}=\beta$. By assumption we can do so by using finitely
many cyclotomic polynomials. The resulting power series be
$W_1$. The inverse of the product of finitely many cyclotomic
polynomials is a power series with poles at certain roots of unity,
hence, we can express the sequence of coefficients as a polynomial in
$n$ and Ramanujan-sums $c_d(n)$ for $d$ dividing some integer
$q$. Consider some point $(n, m)\in\overline{A_W}$, and compute the
coefficient attached to this point in $W_1$. If $A_W$ does not contain
a point $(n', m')$, such that $(n-n', m-m')$ is collinear to $(\beta,
1)$, then this coefficient is clearly 0. Otherwise we consider all
points $(n_1, m_1), \ldots, (n_k, m_k)$ in $A_W$, which are on the
parallel to $(\beta t, t)$ through $(n, m)$. The coefficients of
$W_1$ attached to points on $\ell$ are linear combinations of shifted
coefficients of inverse cyclotomic polynomials, hence, they can be
written as some polynomial with periodic coefficients. In particular,
either there are only finitely many non-vanishing coefficients, or
there exists a complete arithmetic progression of non-vanishing
coefficients. Hence, we find that $A_{W_1}$ is contained within a
locally finite set of lines parallel to $(\beta t, t)$, and every line 
either contains only finitely many points, or a complete arithmetic
progression. Suppose that every line contains only finitely many
points. Then there exists some $\beta'>\beta$, such that $A_{W_1}$
is contained in $\{(s, t):s\geq\beta_1 t\}$, in particular, $W_1$ is
regular in $\{(z_1, z_2):|z_1|^\beta\leq|z_2|\}$. But
$W_1=\frac{P}{Q\tilde{W}}$, and by assumption $P$ is not divisible by
$\tilde{W}$, therefore, there exist points $(z_1, z_2)$, where
$\tilde{W}$ vanishes, but $P$ does not, and these points are
singularities of $W_1$ satisfying $|z_1|^\beta\leq|z_2|$. Hence,
there exists some line containing a complete arithmetic progression.

Let $(x, 0)+t(\beta, 1)$ be the unique line containing infinitely many
elements of $A_{W_1}$, such that for each $y>0$ the line $(y, 0)+t(\beta,
1)$ contains only finitely many elements $(n_1, m_1), \ldots, (n_k, m_k)$ of
$A_{W_1}$. Set $\delta=-x$, that is, the distance of the right
boundary of $A_W$ from the line $(x, 0)+t(\beta, 1)$ measured
horizontaly, and set $\delta_i=-m_i\beta/n_i$, that is, $\delta_i$ is
the distance of $(n_i, m_i)$ from the right boundary, also measured
horizontally, and $\delta_-=\min\delta_i>0$.

We now eliminate the points $a_i$ to obtain the power series
$W_2$. When doing so we introduce lots of 
new elements to the left of the line $(x, 0)+t(\beta, 1)$, which are
of no interest to us, and finitely many points on this line or to the
right of this line, in fact, we can get points at most at the points of the
form $\lambda(n_i, m_i)+\mu(n_j, m_j)$, $\lambda, \mu\in\N$, $\lambda,
\mu>0$. Note that the horizontal distance from the line $t(\beta, 1)$
is additive, that is, $A_{W_2}$ is contained in the intersection of
$\overline{A_W}$ and the half-plane to the left of the line $(x,
0)+t(\beta, 1)$, together with finitely many points between the lines $(x,
0)+t(\beta, 1)$ and $t(\beta, 1)$, each of which has distance at
least $2\delta_-$ from the latter line. Repeating this procedure, we
can again double this distance, and after finitely many steps this minimal
distance is larger than the width of the strip, which means that we
have arrived at a power series $W_3$ such that $A_{W_3}$ is contained
in the intersection of $\overline{A_W}$ and the half-plane to the left
of the line $(x, 0)+t(\beta, 1)$. Moreover, since at each step there
are only finitely many points changed on the line $(x, 0)+t(\beta, 1)$
, we see that the intersection of $A_{W_3}$ with this line
equals the intersection of $A_{W_1}$ with this line up to finitely
many inclusions or omissions. Since an infinite arithmetic progression, from
which finitely many points are deleted still contains an infinite
arithmetic progression, we see that $A_{W_3}$ contains an infinite
arithmetic progression. 

Next we eliminate the points on $A_W$ starting at the bottom and
working upwards. When eliminating a point, we introduce (possibly
infinitely many) new points, but all of them are on the left of the
line $(x, 0)+t(\beta, 1)$. Hence, after infinitely many steps we
arrive at a power series $W_4$, for which $A_{W_4}$ is contained in
the intersection of $\overline{A_W}$ and the open half-plane to the
left of $(x, 0)+t(\beta, 1)$. 

Fortunately, from this point on we can be less explicit. Consider
the set of differences of the sets $A_{W_i}$ from the line $t(\beta,
1)$. Taking the differences is a semi-group homomorphism, hence, at
each stage the set of differences is contained in the semi-group
generated by the differences we started with. But since $W$ is a
polynomial, this semi-group is finitely generated, and therefore
discrete. Hence, no matter how we eliminate terms, at each stage the
set $A_{W_i}$ is contained in a set of parallels to $t(\beta, 1)$
intersecting the real axis in a discrete set of non-positive numbers. 

Collecting the cyclotomic factors used during this procedure, we have
proven the following.

\begin{Lem}
\label{Lem:elim}
Let $W(X, Y)$ be a rational function such that $\tilde{W(X, Y)}$
is a cyclotomic polynomial, but $W$ itself is not cyclotomic. Define
$\beta$ as above. Then there is a unique
expansion $W(X, Y) = \prod_{n, m} (1-X^n Y^m)^{c_{n, m}}$. The set
$C=\{(n, m): c_{n, m}\neq 0\}$ contains an infinite arithmetic
progression with difference a multiple of $(\beta, 1)$, only finitely
many elements to the right of this line, and all entries are on lines
parallel to $t(\beta, 1)$, such that the lines intersect the real
axis in a discrete set of points.
\end{Lem}

\section{Proof of case 3}
We prove that $\beta$ is an obstructing point. For integers $n, m$ with
$c_{n, m}\neq 0$ the factor $\zeta(-n+ms)^{c_{n, m}}$ creates a pole
or a zero
at $\frac{n+1}{m}$, which for $\frac{n+1}{m}>\beta$ is to the right of
the supposed boundary. Hence, if $\beta$ is not an obstructing point,
 for some $\epsilon>0$ and all rational numbers $\xi\in(\beta, 
 \beta+\epsilon)$ we would have  $\sum_{\frac{n+1}{m}=\xi} c_{n, m} = 0$. We now
show that this is impossible by proving that there are pairs $(n, m)$
with $\frac{n+1}{m}$ arbitrarily close to $\beta$, $c_{n, m}\neq 0$,
such that the sum consists of a single term, and is therefore non-zero
as well. 

Let $\frac{k}{\ell}$ be the slope of the rays. 
Let $\{(n_i, m_i)\}$ be a list of the starting points of the rays
described in Lemma~\ref{Lem:elim}, where $(n_0, m_0)$ defines the
right-most ray.
Take an integer
$q$, such that $c_{km\nu+n_0, \ell q\nu+m_0}\neq 0$ for all but
finitely many natural numbers $\nu$. Let $d$ be the greatest common divisor of 
$m_0$ and $q$. The prime number theorem for arithmetic progressions guarantees
infinitely many $\nu$, such that $\frac{\ell\nu+m_0}{d}=p$ is
prime. Suppose there is a pair $n', m'$ belonging to another ray, such
that $c_{n', m'}\neq 0$ and $\frac{n'+1}{m'}=
\frac{k\nu+n_0}{\ell\nu+m_0}$. The point $(n', m')$ must lie on one
of the finitely many rays, hence, we can write $n'=k\nu'+n_1$,
$m'=\ell\nu'+m_1$. Since $p$ is a divisor of the denominator of the
right hand side, it also has to divide the denominator of the left
hand side. We obtain that $p$ divides both $\ell\nu+m_0$ and
$\ell\nu'+m_1$. Restricting, if necessary, to an arithmetic
progression, we obtain an infinitude of indices such that
$\ell\nu'+m_1=t(\ell\nu +m_0)$, where $t\in[0, 1]$ is a rational number with
denominator dividing $d$. Hence, we obtain that the equations
\[
\ell\nu'+m_1=t(\ell\nu +m_0),\quad t(k\nu'+n_1) = k\nu+n_0
\]
have infinitely many  solutions $\nu, \nu' \in \N $. Two linear
equations in two variables, none of which is trivial, can only have
infinitely many solutions, if these equations are equivalent, that is, $t^2=1$,
which implies $t=1$ since $t$ is positive by definition. Hence, writing the
equations as vectors, we have
\[
(\nu-\nu')\binom{k}{\ell} = \binom{n_1}{m_1} - \binom{n_0}{m_0},
\]
that is, the vector linking $\binom{n_0}{m_0}$ with
$\binom{n_1}{m_1}$ is collinear with $\binom{k}{\ell}$, contrary
to the assumption that $n', m'$ was on a ray other than that of $n, m$. Hence,
poles of $\zeta$-factors accumulate at $\beta$. It remains to check
that these poles are not cancelled by zeros of other factors. Since
zeros of $\zeta$-factors are never positive reals, these factors do not
cause problems. Suppose that a pole of $\zeta(ns-m)$ cancels with a
zero of the local factor $W(p, p^{-s})$, that is, $W(p,
p^{-(m+1)/n})=0$. Since $W$ has coefficients in $\Z$, this implies
that $p^{-(m+1)/n}$ is algebraic of degree at most equal to the degree
of $W$, hence, $\frac{m+1}{n}$ can be reduced to a fraction with
denominator at most equal to the degree of $W$. There are only
finitely many rational numbers in the interval $[\beta, \beta+1]$
with bounded denominator, hence, only finitely many of the poles
can be cancelled, that is, $\beta$ is in fact an obstructing point.

For the corollary note that in cases (2)--(4) $\beta$ is an
obstructing point, that is, in these cases the first condition of the
corollary holds true. In case (1) and (5), we can represent $D$ as the
product of finitely many Riemann $\zeta$-functions multiplied by some
function which is holomorphic in the half-plane
$\Re\;s>\beta$, and has zeros only where the finitely many local
factors vanish. A local factor belonging to the prime $p$ creates a
$\frac{2\pi i}{\log p}$-periodic pattern of zeros, hence, the number
of zeros and poles is bounded above by the number of zeros of the
finitely many $\zeta$-functions, which is $\mathcal{O}(T\log T)$, and
the finitely many sets of periodic patterns, which create
$\mathcal{O}(T)$ zeros. Hence, $N_\pm(\{\Re z>\beta, |\Im z|<T\})$ is
$\mathcal{O}(T\log T)$. It may happen that there are significantly
less poles or zeros, if poles of one factor coincide with poles of
another factor, however, we claim that under RH and the assumption of
linear independence of zeros the amount of cancellation is
negligible. First, if the imaginary part of zeros of $\zeta$ are
$\Q$-linearly independent, then we cannot have
$\zeta(n_1s-m_1)=\zeta(n_2s-m_2)=0$ for integers $n_1, n_2, m_1, m_2$
with $(n_1, m_1)\neq (n_2, m_2)$, that is, zeros and poles of
different $\zeta$-factors cannot cancel. There is no cancellation among
local factors, since local factors can only have zeros and never
poles. Now consider cancellation among zeros of local factors and
$\zeta$-factors. We want to show that there are at most finitely many
cancellations. Suppose otherwise. Since there are only finitely many
local factors and finitely many $\zeta$-factors, an infinitude of
cancellation would imply that there are infinitely many cancellations
among one local factor and one $\zeta$-factor. The zeros of a local
factor are of the form $\xi_i+\frac{2k\pi i}{\log p}$, where
$\xi_i$ is the logarithm of one of the roots of $W(p, X)=0$ chosen
in such a way that $0\leq\Im\xi_i<\frac{2\pi i}{\log p}$. Since an
algebraic equation has only finitely many roots, an infinitude of
cancellations implies that for some complex number $\xi$ and
infinitely many integers $k$ we have $\zeta(n(\alpha+\frac{2k\pi
  i}{\log p})-m)=0$. Choose 4 different such integers $k_1, \ldots,
k_4$, and let $\rho_1, \ldots, \rho_4$ be the corresponding roots of
$\zeta$. Then we have $\rho_1-\rho_2=\frac{2(k_1-k_2)n\pi}{\log p}$,
$\rho_3-\rho_4=\frac{2(k_3-k_4)n\pi}{\log p}$, that is,
$(k_3-k_4)(\rho_1-\rho_2)=(k_1-k_2)(\rho_3-\rho_4)$, which gives a
linear relation among the zeros of $\zeta$, contradicting our assumption.
Hence, if the imaginary pars of the roots of $\zeta$ are $\Q$-linear
independent, the number of zeros and poles of $D$ in some domain
coincides with the sum of the numbers of zeros and poles of all
factors, up to some bounded error, and our claim follows.

\section{Examples}
In this section we give examples to show that our classification is
non-trivial in the sense that every case actually occurs.

\begin{ex}
 The sum $\sum_{n=1}^\infty\frac{\mu^2(n)\sigma(n)}{n^s}=
\frac{\zeta(s)\zeta(s-1)}{\zeta(2s)\zeta(2s-2)}$ corresponds to the
polynomial $W(X, Y)=(1+Y)(1+XY)$, while the sum 
$\sum_{n=1}^\infty\frac{\sigma(n)}{n^s}=
\zeta(s)\zeta(s-1)$ corresponds to the rational function $W(X,
Y)=\frac{1}{(1+Y)(1+XY)}$. 
\end{ex}

\begin{ex}

{\bf (a)}  Let $\Omega(n)$ be the number of prime divisors of $n$ counted with
multiplicity. Then
$\sum_{n=1}^\infty\frac{2^{\Omega(n)}}{n^s}=\prod_p(1+\frac{2}{p^s}(1-p^{-s})))$
corresponds to the rational function $W(X, Y)=1+\frac{2Y}{1-Y}$ with
main part $1+2Y$, which is not cyclotomic. 

\smallskip\noindent
{\bf (b)}  Let $G$ be the direct
product of three copies of the Heisenberg-group,
$a_n^\triangleleft(G)$ the number of normal subgroups of $G$ of index
$n$. Then $\zeta_G^\triangleleft(s) = \sum_{n=1}^\infty
\frac{a_n^\triangleleft(G)}{n^s}$ was computed by Taylor\cite{Taylor}
and can be written as a finite product
of $\zeta$-functions and an Euler-product of the form $\prod_p W(p,
p^{-s})$, where $W$ consists of 14 monomials and $\tilde{W}(X, Y) =
1-2X^{13}Y^8$, which is not cyclotomic.

\end{ex}
\begin{ex}
 {\bf (a)}  Let $G$ be the free nilpotent group of class two with three
generators. Then $\zeta_G^\triangleleft(s)$ can be written as a finite
product of $\zeta$-functions and the Euler-product $\prod_p W(p,
p^{-s})$, where
\[
W(X, Y) = 1+X^3Y^3 + X^4Y^3+X^6Y^5+X^7Y^5+X^{10}Y^8.
\]
We have $\tilde{W}(X, Y)=1+X^7Y^5$, which clearly does not divide $W$,
hence, while $\tilde{W}$ is cyclotomic, $W$ is not. Hence, $W$ is not
case 1 or 2. Theorem~\ref{thm:class} implies that $7/5$ is an
essential singularity of $\zeta_G^\triangleleft$. Du Sautoy and
Woodward\cite{book} showed that in fact the line $\Re\;s=7/5$ is the
natural boundary for $\zeta_G^\triangleleft$.

\smallskip\noindent
{\bf (b)}  Now consider the product
\begin{eqnarray*}
f(s) = \prod_p \Big(1+p^{-s} + p^{1-2s}\Big)
\end{eqnarray*}
Again, the polynomial $W(X, Y)=1+Y+XY^2$ is not cyclotomic, while
$\tilde{W}$ is cyclotomic. Again, Theorem~\ref{thm:class} implies that
$1/2$ is an obstructing point of $f$. However, the question
whether there exists another point on the line $\Re\;s=1/2$ which is an
obstructing point is essentially equivalent to the Riemann hypothesis. We
have 
\begin{eqnarray*}
f(s) & = & \frac{\zeta(s)\zeta(2s-1)\zeta(3s-1)}{\zeta(2s)\zeta(4s-2)} R(s)\\
&&\times\;\prod_{m\geq1} \frac{\zeta((4m+1)s-2m)}
{\zeta((4m+3)s-2m-1)\zeta((8m+2)s-4m)},
\end{eqnarray*}
hence, if $\zeta$ has only finitely many zeros off the line $1/2+it$,
then the right hand side has only finitely many zeros in the domain
$\Re\;s>1/2$, $|\Im\;s|>\epsilon$, hence, $1/2$ is the unique obstructing
point on this line. On the other hand, if $\zeta(s)$ has infinitely
many non-real zeros off the line $1/2+it$, then every point on this
line is an obstructing point for $f$ (confer \cite{BSP}).

Hence, while for some polynomials the natural boundary can be
determined, we do not expect any general progress in this case.
\end{ex}
\begin{ex}
{\bf (a)}  The local
zeta function associated to the algebraic group $\mathcal G$  is defined as 
$$
Z_p(\mathcal G, s)=\int_{\mathcal G_p^+} \mid  \det(g)\mid_p^{-s}d\mu
$$
where $\mathcal G_p^+=G(\Q_p)\cap M_n (\Z_p)$ , $\mid.\mid_p$ denotes the 
p-adic valuation
and $\mu$ is the normalised Haar measure on $ \mathcal G(\Z_p)$.
In particular
 the zeta function associated to the group $\mathcal G=GSp_6$\cite{Igusa}
is given by
\begin{eqnarray*}
Z(s/3) = \zeta(s)\zeta(s-3)\zeta(s-5)\zeta(s-6)\prod_p
\Big(1+p^{1-s}+p^{2-s}+p^{3-s}+p^{4-s}+p^{5-2s}\Big).
\end{eqnarray*}
The polynomial 
$$W(X, Y)=1+(X+X^2+X^3+X^4)Y+X^5Y^2$$ 
satisfies the relation
$\tilde{W}(X, Y)=1+X^4Y$, that is, $\tilde{W}$ is cyclotomic, while $W$
is not. Du Sautoy an Grunewald\cite{ghost} showed that in the
cyclotomic expansion of $W$ there are only finitely many $(n, m)$ with
$c_{n, m}\neq 0$ and $\frac{n+1}{m}>4$, and that $W(p, p^{-s})=0$ has
solutions with $\Re\;s>4$ for infinitely many primes, hence, $W$ is an
example of type 4, and $Z(s/3)$ has the natural boundary $\Re\;s=4$.

\smallskip\noindent
{\bf (b)} Let $V$ be the cubic variety $x_1x_2x_3=x_4^3$, $U$ be the open subset
$\{\vec{x}\in V\cup\Z^4: x_4\neq 0\}$, $H$ the usual height
function. De la Bret\`eche and Sir Swynnerton-Dyer\cite{2nob} showed
that $Z(s)=\sum_{x\in U} H(x)^{-s}$ can be written as the product of
finitely many $\zeta$-functions, a function holomorphic in a
half-plane strictly larger than $\Re\;s>3/4$, and a function having an
Euler-product corresponding to the rational function
\[
W(X, Y) = 1 + (1-X^3Y)(X^6Y^{-2}+X^5Y^{-1}+X^4+X^2Y^2+XY^3+Y^4)-X^9Y^3.
\]
They showed that in the cyclotomic expansion of this function there
occur only finitely many terms $c_{n, m}X^nY^m$ with $c_{n, m}\neq 0$
and $\frac{n+1}{m}>\frac{3}{4}$, and all but finitely many local
factors have a zero to the right of $\Re\;s=3/4$, hence, $\Re\;s=3/4$
is the natural boundary of $Z(s)$.
\end{ex}
\begin{ex}
 Let $J_2(n)$ be the Jacobsthal-function, i.e. $J_2(n)=\#\{(x,
y):1\leq x, y\leq n, (x, y, n)=1\}$, and define $g(s)=\sum_{n\geq
  1}\frac{\mu(n)J_2(n)}{n^2}$. Since $J_2$ is multiplicative, $g$ has
an Euler-product, which can be computed to give
\begin{eqnarray*}
g(s) = \prod_p \Big(1+p^{-s} -p^{2-s}\Big).
\end{eqnarray*}
We have
\[
g(s) = \prod_p(1-p^{2-s})\prod_p(1+\frac{p^{-s}}{1-p^{2-s}}) =
\zeta(s-2)D^*(s), 
\]
say. For $\sigma=\Re\;s>2+\epsilon$ the Euler product for $D^*$ converges
uniformly, since
\[
\sum_p \big|\frac{p^{-s}}{1-p^{2-s}}\big| \leq \sum_p
\frac{p^{-\sigma}}{1-2^{2-\sigma}}\leq \frac{\zeta(2)}{\epsilon}.
\]
Hence, $D^*$ is holomorphic and non-zero in $\Re\;s>2$, that is, no
point on the line $\Re\;s=2$ is an obstructing point, that is,
Estermann's method cannot prove the existence of a single singularity
of this function.
\end{ex}
\section{Comparison of our classification with the classification of
  du Sautoy and Woodward}

In \cite{book}, du Sautoy and Woodward consider several classes of
polynomials for which they can prove Conjecture 1. Since their classes
do not coincide with the classes described in Theorem~\ref{thm:class},
we now describe how the two classifications compare. We will refer to the
classes described in Theorem~\ref{thm:class} as  `cases', while we will
continue
to refer to the polynomials of du Sautoy and Woodward by their original
appellation of  ` type'.

Polynomials of type I are polynomials $W$ such that $\tilde{W}$ is not
cyclotomic, this class coincides with polynomials in case (2).

Polynomials of type II are polynomials $W$ such that $\tilde{W}$ is
cyclotomic, there are only finitely many $c_{n, m}> 0$ with
$\frac{n+1}{m}>\beta$, and for infinitely many primes we have that
$W(p, p^{-s})$ has zeros to the right of $\beta$. This class contains
all polynomials in case (4), and all polynomials of type II fall under
case (3) or (4), but there are polynomials in case (3) which are not
of type II\@. For polynomials of type II they prove that the line
$\Re\;s=\beta$ is the natural boundary of meromorphic continuation of
$D$,  their result for polynomials therefore clearly supersedes
the relevant parts of Theorem~\ref{thm:class}. 

Polynomials of type III are polynomials $W$ as in type II, but there
are infinitely many pairs $n, m$ with $c_{n, m}>0$,
$\frac{n+1}{m}>\beta$. These polynomials fall under case (3), they
show under the Riemann hypothesis that $\Re\;s=\beta$ is a natural
boundary. For such polynomials the results are incomparable, our
results are unconditional, yet weaker. 

Polynomials of type IV are polynomials with infinitely many pairs $(n,
m)$ satisfying $c_{n,m}\neq 0$ and $\frac{n+1/2}{m}>\beta$, and such
that with the exception of finitely many $p$ there are no local zeros
to the right of $\Re\;s=\beta$. For such 
polynomials du Sautoy and Woodward show that $\Re\;s=\beta$ is the
natural boundary, if the imaginary parts of the zeros of $\zeta$ are
$\Q$-linearly independent. All polynomials of type IV fall under case
(3), again, the results are incomparable.

Polynomials of type V are polynomials $W$ such that $\tilde{W}$ is
cyclotomic, with the exception of finitely many $p$ there are no local
zeros to the right of $\beta$, and 
there are only finitely many pairs $n, m$ with $c_{n, m}\neq 0$ and
$\frac{n+1}{m}\geq\beta$. This correspond to case (5).

Polynomials of type VI are polynomials $W$ such that $\tilde{W}$ is
cyclotomic, with the exception of finitely many $p$ there are no local
zeros to the right of $\beta$, there are infinitely many pairs $(n,
m)$ with $c_{n, m}\neq 0$ and $\frac{n+1}{m}>\beta$, only finitely
many of which satisfy $\frac{n+1/2}{m}>\beta$. These fall under case
(3).

Case (1) does not occur in their classification as it is justly
regarded as trivial.

\section{Comparison with the multivariable case}

The object of our study has been the Dirichlet-series $D(s)=\prod W(p,
p^{-s})$. This will be called the
$1\frac{1}{2}$-variable problem since the polynomial has two variables, but the
Dirichlet-series depends on only one complex variable. If  the
coefficients of the above series have some arithemetical meaning,
and this meaning translates into a statement on each monomial of $W$, then
the Dirichlet-series $D(s_1, s_2)=\prod_p W(p^{-s_1}, p^{-s_2})$ retains more
information, and it could be fruitful to consider this function instead. Of
course, the gain in information could be at the risk of the technical
difficulties introduced by considering several variables. However, here we
 show that the multivariable problem is actually easier then the original
question of $1\frac{1}{2}$-variables. 

Where there is no explicit reference to $p$, the problem of a natural boundary
was completely solved by Essouabri, Lichtin and the first named author\cite{Forum}. 
\begin{Theo}
Let $W\in\Z[X_1, \ldots, X_k]$ be a polynomial satisfying $W(0, \ldots, 0) =
1$. Set $D(s_1, \ldots, s_k)=\prod_p W(p^{-s_1}, \ldots, p^{-s_k})$. Then $D$
can be meromorphically continued to the whole complex plane if and only if $W$
is cyclotomic. If it cannot be continued to the whole complex plane, then its
maximal domain of meromorphic continuation is the intersection of a finite
number of effectively computable half-spaces. The bounding hyper-plane of each
of these half-spaces passes through the origin.
\end{Theo}
 
At first sight one may think that one can pass from the 2-dimensional by
fixing $s_1$, however, this destroys the structure of the problem, as is
demonstrated by the following.

\begin{ex}
\label{ex:2to1}
The Dirichlet-series $D(s_1, s_2) = \prod_p 1+(2-p^{-s_1})p^{-s_2}$  as a
function of two variables can be
meromorphically continued into the set $\{(s_1, s_2): \Re\;s_2>0,
\Re\;s_1+s_2>0\}$, and the boundary of this set is the natural boundary of
meromorphic continuation. If we fix $s_1$ with $\Re\;s_1\geq 0$, and view $D$
as a function of $s_1$, then $D$ can be continued to $\C$ if and only if
$s_1=0$. In every other case the line $\Re\;s_2=0$ is the natural boundary.
\end{ex}
\begin{proof}
The behaviour of $D(s_1, s_2)$ follows from \cite[Theorem 2]{Forum}. 
If we fix $s_1$,
then $1+(2-p^{-s_1})p^{-s_2}$ has zeros with relatively large real part,
provided that either $\Re s_1>0$, or $\Re s_1=0$ and $\Re p^{-s_1}<0$. In the
first case we can argue as in the case that $\tilde{W}$ is not cyclotomic. By
the prime number theorem for short intervals we find that the number of prime
numbers $p<x$ satisfying $\Re p^{-s_1}<0$ is greater than $c\frac{x}{\log x}$, and we see
that we can again adapt the proof for the case $\tilde{W}$ non-cyclotomic. 
\end{proof}
In other words, the natural boundary for the $1\frac{1}{2}$-variable problem
is the same as for the 2-variable problem, with one exception, in which the
$1\frac{1}{2}$-variable problem collapses to a 1-variable problem, and in
which case the Euler-product becomes continuable beyond the 2-variable
boundary. 

It seems likely that this behaviour should be the prevalent one, it is less
clear what precisely ``this behaviour'' is. One quite strong possibility is
the following:

{\em Suppose that $D(s_1, s_2)=\prod_p W(p^{-s_1}, p^{-s_2})$ has a natural
  boundary at $\Re\;s_1=0$. Then there are only finitely many values $s_2$, for
  which the specialization $D(\cdot, s_2)$ is meromorphically continuable
  beyond $\Re\;s_1=0$.}

However, this statement is right now supported only by a general lack of
examples, and the fact that example~\ref{ex:2to1} looks quite natural,
so we do not dare a conjecture. However we believe that some progress in
this direction could be easier to obtain than directly handling
Conjecture~\ref{Con:Main}. In particular those cases, in which zeros of
$\zeta$ pose a serious threat for local zeros would become a lot easier since this
type of cancellation can only affect a countable number of values for $s_2$.

\begin{tabular}{ll}
Gautami Bhowmik, & Jan-Christoph Schlage-Puchta,\\
Universit\'e de Lille 1, & Albert-Ludwigs-Universit\"at,\\
Laboratoire Paul Painlev\'e, & Mathematisches Institut,\\
U.M.R. CNRS 8524,  & Eckerstr. 1,\\
  59655 Villeneuve d'Ascq Cedex, & 79104 Freiburg,\\
  France & Germany\\
bhowmik@math.univ-lille1.fr & jcp@math.uni-freiburg.de
\end{tabular}

\end{document}